\documentclass[reqno,a4paper,12pt]{amsart}
\usepackage{a4wide}

\usepackage{tikz}
\usepackage{pgflibraryarrows}
\usepackage{pgflibrarysnakes}

\usepackage{amsfonts}
\usepackage{amssymb}
\usepackage{amsthm}
\usepackage{amsmath}
\usepackage[all]{xy}

\usepackage{pb-diagram}

\SelectTips{cm}{}

\theoremstyle{plain}
\newtheorem{thm}{Theorem}[section]

\newtheorem{lem}[thm]{Lemma}
\newtheorem{conj}[thm]{Conjecture}

\theoremstyle{definition}
\newtheorem{defn}[thm]{Definition}

\newtheorem{ques}[thm]{Question}

\theoremstyle{remark}
\newtheorem{rem}[thm]{Remark}
\newtheorem{expl}[thm]{Example}

\numberwithin{equation}{section}

\newcommand{\sL}{\mathcal{L}}
\newcommand{\sM}{\mathcal{M}}
\newcommand{\sN}{\mathcal{N}}
\newcommand{\sS}{\mathcal{S}}

\newcommand{\divi}{\textup{div}}

\newcommand{\BSO}{\textup{BSO}}
\newcommand{\BSTOP}{\textup{BSTOP}}

\newcommand{\LL}{\mathbb{L}}
\newcommand{\NN}{\mathbb{N}}
\newcommand{\RR}{\mathbb{R}}
\newcommand{\QQ}{\mathbb{Q}}
\renewcommand{\SS}{\mathbb{S}}
\newcommand{\ZZ}{\mathbb{Z}}

\newcommand{\id}{\textup{id}}
\newcommand{\im}{\textup{im}}

\newcommand{\Wh}{\textup{Wh}}

\newcommand{\ra}{\rightarrow}
\newcommand{\xra}{\xrightarrow}

\newcommand{\co}{\colon\!}

\newcommand{\OO}{\textup{O}}
\newcommand{\G}{\textup{G}}
\newcommand{\TOP}{\textup{TOP}}

\renewcommand{\int}{\textup{int}}

\newcommand{\pr}{\textup{pr}}

\newcommand{\colim}{\textup{colim}}

\newcommand{\asmb}{\textup{asmb}}

\newcommand{\BO}{\textup{BO}}

\newcommand\sAut{\mathrm{sAut}}
\newcommand\Diff{\mathrm{Diff}}
\newcommand\wt{\widetilde}

\newcommand\C{\mathbb{C}}

\newcommand\R{\mathbb{R}}

\newcommand\Q{\mathbb{Q}}
\newcommand\Z{\mathbb{Z}}

\newcommand{\an}[1]{\langle{#1}\rangle}

\newcommand\hAut{\mathrm{hAut}}
\newcommand\Lone{\mathbb{L}\an{1}}

\newcommand\coker{\mathrm{coker}}
\newcommand\del{\partial}

\usepackage{color}

\title[]{On the cardinality of the manifold set}

\author{Diarmuid Crowley and Tibor Macko}

\date{\today}



\address{School of Mathematics \& Statistics,
The University of Melbourne,
Parkville, VIC, 3010,
Australia}
\email{dcrowley@unimelb.edu.au}

\address{
Mathematical Institute, Slovak Academy of Sciences, \v Stef\'anikova 49, Bratislava, SK-81473, Slovakia}
\email{macko@mat.savba.sk}

\thanks{The second author was supported by the project ``Topology of high-dimensional manifolds'' under the scheme ``Returns'' of the ``Ministry of education, science, research and sport of the Slovak republic'', by the grant VEGA 1/0101/17 and by the Slovak Research and Development Agency under the contract No. APVV-16-0053.}

\begin{document}

\maketitle

\begin{abstract}
We study the cardinality of the set of homeomorphism classes of manifolds homotopy equivalent to a given manifold $M$ and compare it to the cardinality of the structure set of $M$.
\end{abstract}

\section{Introduction} \label{sec:intro}
Let $M$ be a closed connected topological $n$-manifold.
We define its {\em manifold set} $\sM(M)$,
to be the set of homeomorphism classes of manifolds homotopy equivalent to $M$:
$$ \sM(M) := \{ N \, | \, N \simeq M \}/{\cong}$$
Here $\simeq$ denotes a homotopy equivalence and $\cong$ denotes a homeomorphism.
We also define the {\em simple manifold set} $\sM^s(M) \subset \sM(M)$ where
we require that $N$ is simple homotopy equivalent to $M$ and
the {\em $s/h$-manifold set}, $\sM^{s/h}(M)$, which is the quotient
of $\sM^s(M)$ where we replace homeomorphism by the relation of $h$-cobordism.
{\em Henceforth, for simplicity, we assume that $M$ is orientable.}

We have the inclusion $\sM^s(M) \hookrightarrow \sM(M)$ and the surjection $\sM^s(M) \to \sM^{s/h}(M)$,
both of which are bijections if the Whitehead group $\Wh(\pi_1(M))$ vanishes and $n \geq 5$.
In Section~\ref{sec:surgery-preliminaries} we give more details on various versions of the manifold set and relations between them.

The computation of the simple manifold set typically proceeds via surgery theory
where one first computes the {\em simple structure set} of $M$,
$$ \sS^{s}(M) := \{ f \colon N \to M \}/{\sim_{s}},$$
which is the set of simple homotopy equivalences from a closed topological $n$-manifold $N$ to $M$,
up to $s$-cobordism in the source (see Section \ref{sec:surgery-preliminaries} for more details).
The simple structure set of $M$ with $\pi = \pi_{1} (M)$ and $n \geq 5$ lies in the surgery exact sequence
\begin{equation} \label{ses:intro}
\cdots \xra{~~~} L^{s}_{n+1}(\ZZ\pi) \xra{~\omega~} \sS^{s}(M) \xra{~\eta~} \sN(M) \xra{~\theta~} L^{s}_n(\ZZ\pi),
\end{equation}
which we recall in Section~\ref{sec:surgery-preliminaries} and which is an exact sequence of abelian groups and homomorphisms.
%
%
%
%
Thanks to the $s$-cobordism theorem the simple structure set maps onto the simple manifold set by the forgetful map
$$ \sS^s(M) \to  \sM^s(M), \quad
[f \colon N \to M] \mapsto [N],$$
which descends to define a bijection
$$ \sS^{s}(M)/\sAut(M) \xra{~\equiv~} \sM^s(M),$$
where $\sAut(M)$, the group of homotopy classes of simple homotopy automorphisms of
$M$, acts on the structure set via post-composition,
$$ \sS^{s}(M) \times \sAut(M) \to \sS^{s}(M), \quad
([f \colon N \to M], [g]) \mapsto [g \circ f \colon N \to M].$$
There is also a similar $h$-decorated surgery exact sequence where
the structure set $\sS^h(M)$ consists of $h$-cobordism classes of homotopy equivalences
$f \colon N \to M$; again see Section \ref{sec:surgery-preliminaries} for details.
When $\Wh (\pi) = 0$ we simply write $\sS(M)$ in place of $\sS^s(M) = \sS^h(M)$.

There are well-developed tools for studying $\sS^{s}(M)$ which have lead to
the computation of $\sS^s(M)$ in many examples.
In this paper we consider the passage from $\sS^{s}(M)$ to $\sM^{s/h}(M)$.
This can be a difficult problem to solve since the group $\sAut(M)$ and
its action on $\sS^{s}(M)$ are in general complicated and the same holds for
the map $\sM^s(M) \to \sM^{s/h}(M)$.
To keep the discussion tractable we
shall focus on the relative cardinalities of $\sS^{s}(M)$ and $\sM^{s/h}(M)$.

For a set $S$, we let $|S|$ denote the cardinality of $S$ and pose the following question:
%

\begin{ques}
	When does $|\sS^{s}(M)| = \infty$ entail
	that $|\sM^{s/h}(M)| = \infty$?
\end{ques}

One way to show that $\sS^s(M)$ and $\sM^{s/h}(M)$ are both
infinite is to find an $h$-cobordism invariant of manifolds, $\gamma(M)$ say, for which
there are infinitely many simple structures $[f_i \colon N_i \to M]$ where
$\gamma(N_i) \neq \gamma(N_j)$ for $i \neq j$.
Important examples of such invariants $\gamma$
are the $\rho$-invariant and its generalisations.
For example, we have
the following mild strengthening of a very general theorem of Chang and Weinberger.
(For a discussion of other relevant work in the literature see Section \ref{ss:further}.)

\begin{thm}[C.f.~{\cite[Theorem 1]{Chang&Weinberger(2003)}}]\label{thm:CW}
Let $M$ have dimension $n = 4k{-}1 \geq 7$ and suppose $\pi = \pi_1(M)$ contains torsion.
Then for every $[f] = [f \colon N \to M] \in \sS^s(M)$, the image of
the orbit $L^s_{4k}(\ZZ\pi)[f] \subset \sS^s(M)$ in $\sM^{s/h}(M)$ is infinite.
\end{thm}

\begin{rem}
Theorem 1 of \cite{Chang&Weinberger(2003)}
is stated using $\sM^s(M)$ but the invariant
which they use
is an $h$-cobordism invariant and so Theorem \ref{thm:CW} follows immediately.
\end{rem}

For the manifolds appearing in Theorem \ref{thm:CW}
the action of $L^s_{n+1}(\Z\pi)$ on $\sS^s(M)$ has infinite orbits
which remain infinite when mapped to $\sM^{s/h}(M)$.
In contrast, our first two main results
concern the case when
the image of $\eta \colon \sS^{h}(M) \to \sN(M)$ is infinite.
In many situations we show that $|\eta(\sS^{h}(M))| = \infty$ implies
that $|\sM^{s/h}(M)| = \infty$.
The invariants $\gamma(N)$ we use are elementary and
are derived from the Hirzebruch $\sL$-classes of $N$ as follows.

Recall that the Hirzebruch $\sL$-polynomial $\sL_k$,
is a degree $k$ rational polynomial in the Pontryagin classes
and let $c_k$ be the least common multiple of the absolute values of the denominators of the coefficients in $\sL_k$.
For smooth manifolds $M_\alpha$ the $\sL$-class $\sL_k (M_\alpha)$ belongs to the lattice
$(1/c_k) \cdot FH^{4k} (M_\alpha;\ZZ) \subset H^{4k} (M_\alpha;\QQ)$,
where $FH^{i} (X;\ZZ) := H^{i} (X;\ZZ)/\textup{tors}$
denotes the free part of the cohomology of a space $X$ and
by a {\em lattice} we simply mean a finitely generated free abelian group.
For topological manifolds there are rational Pontryagin classes
$p_{k} (M) \in  H^{4k} (M;\QQ)$ and we have a rational equivalence of classifying spaces
$\BSO \ra \BSTOP$, see \cite[Annex 3, section 10]{Kirby-Siebenmann(1977)}, \cite{Novikov(1966)}.
For each $i > 0$ fix the smallest positive integer $t_{i} > 0$ such that for any $M$
\[p_i (M) \in (1/t_i) \cdot FH^{4i} (M;\ZZ)\]
and define $t_k$ to be the least common multiple of the $t_i$ with $i < k$.
We define $r_k :=  c_k \cdot t_k$ so that
%
\[\sL_k (M) \in (1/r_k) \cdot FH^{4k} (M;\ZZ) \subset H^{4k} (M;\QQ)\]
%
and for $f \co N \ra M$ representing an element in $\sS (M)$ we define
\[
\divi_{k} (f) := \divi_{k} (\sL_{k} (N)) \in \NN,
\]
where $\divi_{k}(\sL_{k}(N))$ is the smallest natural number $d$
such that $\sL_{k}(N) = d x$ for some class $x \in (1/r_k) \cdot FH^{4k} (M;\ZZ)$;
see Definition~\ref{defn:divisibility-in-lattice}.
In Section~\ref{sec:1-ctd-case} we verify that the sending a manifold structure $[f \colon N \to M]$ to $\divi_k(f)$ induces well-defined
functions $\divi_k \colon \sS(M) \to \NN$
and $\divi_k \colon \sM(M) \to \NN$.

\begin{thm} \label{thmA}
If $\pi_1(M) = \{e\}$ and $n \geq 5$ then
$\sM(M) = \sM^{s/h}(M)$ is infinite if and only if $\sS (M)$
is infinite. In such a case for some $0 < 4k < n$ the set $\divi_{k} (\sM (M))$ is infinite.
\end{thm}


Since the normal invariant map $\eta$ is injective when $M$ is simply-connected,
Theorem \ref{thmA} implies the following statement:
$$ \text{\em If $\pi_1(M) = \{e\}$, $n \geq 5$ and
$|\eta(\sS(M))| = \infty$ then $|\sM^{s/h}(M)| = \infty$.}$$
Our next main result shows that the obvious modification of this statement continues to hold in many situations when
$\pi_1(M) \neq \{e\}$. The proof proceeds by extending the arguments using the divisibility of the Hirzebruch $\sL$-class
from the proof of Theorem \ref{thmA}.

\begin{thm} \label{thmB}
Suppose $n \geq 5$ and $|\eta(\sS^{h}(M))| = \infty$.
Then $|\sM^{s/h}(M)| = \infty$ if any of the following conditions hold:
\begin{enumerate}
\item $M$ is homotopy equivalent to a manifold $M'$ whose stable tangent bundle is trivial;
\item For some $0 < 4k <n$ there exist non-zero sublattices
\[
L \subset L' \subset (1/r_{k}) \cdot FH^{4k} (M; \ZZ)
\]
such that $L \subset \im(\eta)/\textup{tors}$, $L \subset L'$ is of finite index, and $\sL_{k} (M) \in L'$;
\item
The group $\pi = \pi_1(M)$ satisfies $n$-dimensional Poincar\'{e} duality
and for the classifying map $c \co M \ra B\pi$ of the universal cover of $M$ the induced map satisfies
$c_{\ast}([M]) \neq 0 \in H_n(B\pi; \Q)$ and
\[ c_{\ast} \co \bigoplus_{0 < 4k <n} H_{n-4k} (M;\ZZ) \ra \bigoplus_{0 < 4k <n} H_{n-4k} (B\pi;\ZZ) \]
has $|ker (c_{\ast})| = \infty$.
\end{enumerate}
\end{thm}


Given Theorems \ref{thm:CW}, \ref{thmA} and \ref{thmB}
one might wonder whether $\sM^s(M)$ is infinite whenever $\sS^s(M)$ is infinite.
Our third main result, Theorem \ref{thmC} below,
shows this is not the case.

We define manifolds $M_{r,g}$ as follows.  Let $T^r = S^1 \times \dots \times S^1$
be the $r$ torus and let $K \subset T^r$ be the standard $2$-skeleton,
which is the union of all the co-ordinate $2$-torii.
For $k \geq 1$, the complex $K$ embeds into Euclidean space $\R^{4k+3}$.
We let $W_K$ be a regular neighbourhood of such an embedding and
define
$$M_{r, 0} := \del W_K$$
to be the boundary of $W_K$.
For any non-negative integer $g$ we then define
$$M_{r, g} := M_0 \sharp_g (S^{2k+1} \times S^{2k+1}) $$
to be the connected sum of $M_{r, 0}$ and $g$ copies of $S^{2k+1} \times S^{2k+1}$.

To state our results for the manifolds $M_{r, g}$ we use
the following variant of the manifold set
where an identification $\pi_1(M) = \pi$ is fixed.
Let $\sAut_\pi(M) \subset \sAut(M)$ be the subgroup whose base-point preserving representatives
induce the identity on $\pi_1(M)$.  We define
the {\em simple $\pi_1$-polarised manifold set} of $M$ by setting
$$\sM^s_\pi(M) := \sS^s(M)/\sAut_\pi(M)$$
and note there is a natural surjective forgetful map $\sM^s_\pi(M) \to \sM^s(M)$.
Since the group $\pi_{1} (M_{r,g}) \cong \Z^r$ has trivial Whitehead group
\cite{Bass-Heller-Swan(1964)},
we set $\sM_\pi(M_{r, g}) := \sM^s_\pi(M_{r, g})$.

\begin{thm} \label{thmC}
For all $r \geq 3$ and $g \geq r{+}3$,
we have $|\sS(M_{r, g})| = \infty$ but
$|\sM_{\pi}(M_{r, g})| = 1$.
%
\end{thm}

\subsection{The smooth manifold set} \label{ss:smooth}
Of course the manifold set may also be defined in the smooth category.
For a smooth manifold $M_\alpha$ we define its {\em smooth manifold set}
to be the set of diffeomorphism classes of smooth manifolds homotopy equivalent
to $M_\alpha$:
$$ \sM_{\Diff}(M_\alpha) := \{ N_\beta \, | \, N_\beta \simeq M_\alpha \}/{\cong}$$
Here $\cong$ denotes diffeomorphism.
The variations of the smooth manifold set
$\sM^s_{\Diff}(M_\alpha)$, $\sM^{s/h}_{\Diff}(M_\alpha)$
and $\sM^s_{\Diff, \pi}$ are defined analogously to the topological variations.
The simple smooth manifold set is frequently computed as the quotient of
the smooth structure set
$$ \sS^{s}_{\Diff}(M_{\alpha})/\sAut(M_{\alpha}) \xra{~\equiv~} \sM^s_{\Diff}(M_{\alpha}),$$
where $\sS^s_{\Diff}(M_\alpha)$ is the set of
simple homotopy equivalences $f \colon N_\beta \to M_\alpha$ up to
smooth $s$-cobordism.
For example, if $n \geq 5$ and $\Sigma_\alpha \simeq S^n$ is a homotopy $n$-sphere,
then
$$ \sM_{\Diff}(\Sigma_\alpha) = \Theta_n/\{\pm 1\},$$
where $\Theta_n$ is the Kervaire-Milnor group of oriented homotopy $n$-spheres
\cite{Kervaire-Milnor(1963)} and the group $\Z/2 = \{ \pm 1\}$ acts on $\Theta_n$ by reversing orientation.

There is a forgetful map $\sM_{\Diff}(M_\alpha) \to \sM(M)$,
where $M$ is the topological manifold underlying $M_\alpha$
and we also write $\divi_k \colon \sS_{\Diff}(M_\alpha) \to \NN$ for the composition
of the forgetful map and
$\divi_k \colon \sS(M) \to \NN$.
The following result is the smooth analogue of Theorem \ref{thmA}.

\begin{thm} \label{thmD}
If $\pi_1(M_\alpha) = \{e\}$
and $n \geq 5$ then $\sM_{\Diff}(M_\alpha)$ is infinite if and only if $\sS_{\Diff}(M_\alpha)$
is infinite. In such a case for some $0 < 4k < n$ the set $\divi_{k} (\sS_{\Diff}(M_\alpha))$ is infinite.
\end{thm}

Theorems \ref{thmA} and \ref{thmD} show for simply-connected
smooth $n$-manifolds with $n \geq 5$ that
$\sM_{\Diff} (M_\alpha)$ is infinite if and only if $\sS(M)$ is infinite.
We believe this holds more generally.

\begin{conj} \label{conjA}
Let $M_\alpha$ be a smooth $n$-manifold with $n \geq 5$.
Then $|\sM_{\Diff}(M_\alpha)| = \infty$ if and only if $|\sM(M)| = \infty$.
\end{conj}

We turn our attention back to the manifolds $M_{r, g}$.
Note that in the definition of $M_{r, g}$,
if we take the regular neighbourhood $W_K$ to be a smooth regular neighbourhood then
the manifolds $M_{r, g}$ acquire smooth structures,
which we denote by $M_{r, g, \alpha}$.
For the following theorem, which is a smooth refinement of Theorem \ref{thmC},
recall that $bP_{n+1} \subset \Theta_n$ denotes the subgroup of homotopy
$n$-spheres bounding parallelisable manifolds.

\begin{thm} \label{thmE}
Let $s := C^r_2$.
For all $r \geq 3$ and $g \geq r{+}3$,
we have $|\sS^s_{\Diff}(M_{r, g, \alpha})| = \infty$
but
\begin{enumerate}
\item
$|\sM^s_{\Diff, \pi}(M_{r, g, \alpha})| = 1$, if $k = 1$;
\item
$|\sM^s_{\Diff, \pi}(M_{r, g, \alpha})|
 \leq (|\Theta_{4k+2}| + r|\Theta_{4k+1}/bP_{4k+2}| + s |\Theta_{4k}|) < \infty$, if $k \geq 2$.
\end{enumerate}
\end{thm}


\subsection{Further discussion} \label{ss:further}

In this subsection we briefly mention other work related to manifold sets and pose two questions.
This discussion is by no means an exhaustive review of the literature
relevant to manifold sets.


In \cite{Kreck&Lueck(2009)} Kreck and L\"uck studied various forms of rigidity for non-aspherical manifolds and defined a {\em Borel manifold} to be a manifold $M$ such that
for every homotopy equivalence $f \colon N \to M$ there is
a homeomorphism $h \colon N \to M$ which induces the same map on fundamental groups
as $f$.
In the notation of this paper, if $\Wh(\pi_1(M)) = 0$ then $M$ is Borel if and only if $|\sM_\pi(M)| = 1$.
In particular,
Theorem \ref{thmC} states that $M_{r, g}$ is Borel when $r \geq 3$ and $g \geq r{+}3$.
Kreck and L\"{u}ck found many examples of non-aspherical Borel manifolds and identified
a general criterion \cite[Theorem 0.21]{Kreck&Lueck(2009)} for a manifold $M$ to be Borel.
This criterion led us to think about $\im(\eta)$ in this paper.



We have already mentioned the work of Chang and Weinberger~\cite{Chang&Weinberger(2003)}. Recently, in \cite{Weinberger-Yu(2015)} Weinberger and Yu extended the results of \cite{Chang&Weinberger(2003)}. Working with
{\em oriented} manifolds $M$,
they defined the {\em reduced structure group of $M$}, $\wt \sS(M)$,
to be the quotient of $\sS(M)$ by the group generated
by elements of the form
$[f \colon N \to M] - [g \circ f \colon N \to M]$,
where $g \colon M \to M$ is an orientation preserving self homotopy
equivalence; i.e.~$\wt \sS(M)$ is the group of co-invariants of the action
of orientation preserving homotopy self homotopy equivalences on $\sS(M)$.
Defining the {\em oriented manifold set of $M$,} $\sM^+(M)$,
in the obvious way, we have the following diagram of implications
$$
|\wt \sS(M)| = \infty~~ \Longrightarrow~~
|\sM^+(M)| = \infty ~~\Longleftrightarrow~~
|\sM(M)| = \infty.
$$
Under general conditions on $\pi_1(M)$, including when $\pi_1(M)$ is a
non-trivial finite group Weinberger and Yu proved that
$\wt \sS(M)$ is infinite \cite[Theorem 3.9]{Weinberger-Yu(2015)}
if $n = 4k-1 \geq 7$.
See Example \ref{expl:Wg} for an elementary example
where $|\wt \sS(M)| < \infty$ but $|\sM(M)| = \infty$.
Even more recently Weinberger, Xie and Yu~\cite{Weinberger-Xie-Yu(2017)} studied
the group of coinvariants of a different but related action of
the group or orientation preserving self-homotopy equivalences
on $\sS(M)$ (it arises via the surgery composition formula of Ranicki \cite{Ranicki(2009)})  and prove that this group of coinvariants is infinite
if $\pi_1(M)$ contains torsion and
if $n = 4k-1 \geq 7$.


In \cite{Khan(2017)} Khan studied the manifold set in
the case $M = S^{1} \times L$,
where $L$ belongs to a certain class of lens spaces and showed
that $|\sM(S^1 \times L)| = \infty$, see \cite[Theorem 1.7]{Khan(2017)}.
Also Jahren and Kwasik studied the manifold set of $S^{1} \times \RR P^{n}$
and showed that $\sM(S^1 \times \R P^{n})$ is infinite if and only if
$n \equiv 3$~mod~$4$, \cite[Theorem 1.4]{Jahren-Kwasik(2008)}.


The action of the $L$-group on the structure set also often
has infinite orbits when $N = N_0 \sharp N_1$ is the connected sum
of manifolds, each of which has non-trivial fundamental group.
If $n \geq 3$ then $\pi = \pi_1(N)$ is the amalgamated product
of $\pi_1(N_0)$ and $\pi_1(N_1)$ and this frequently leads large UNil-terms in the
$L$-groups.
In \cite{Brookman-Davis-Khan(2007)} Brookman, Davis and Khan computed the
(oriented) manifold set of $\RR P^{n} \# \RR P^{n}$ and showed in particular that it is infinite
if $n \geq 4$ is even, with the infinite size being due to the action of UNil-terms
in the $L$-group.
An interesting feature of UNil-groups is that they are infinitely generated whenever they
are non-zero.
In contrast to this, the normal invariants are always finitely generated since they
are computed by a generalised (co)-homology theory with finitely generated coefficients
in each dimension.


In the simply-connected case, Madsen Taylor and Williams \cite{Madsen-Taylor-Williams(1980)} studied the subset of the manifold set of $M$ defined by the manifolds which are tangentially homotopy equivalent to $M$.
This tangential manifold set is finite and the authors were concerned with estimating its size.

We conclude with two open questions related to the manifold set,
the second of which was posed to us by Shmuel Weinberger.

\begin{ques}
Are there manifolds $M$ with $|\eta(\sS^s(M))| = \infty$ but $|\sM^s(M)| < \infty$?
\end{ques}

\begin{ques}
Are there $(2k{+}1)$-dimensional manifolds $M$
where the orbits of the action of $L^s_{2k+2}(\Z\pi)$ on $\sS^s(M)$ are infinite
but have finite images in $\sM^s(M)$?
\end{ques}






% ...


The rest of this paper is organised as follows. In Section~\ref{sec:surgery-preliminaries} we
review the surgery theory we need later.
Section~\ref{sec:1-ctd-case} contains the proofs of Theorem~\ref{thmA}
and~\ref{thmD},
Section~\ref{sec:non-1-ctd-case} the proof of Theorem~\ref{thmB} and Section~\ref{sec:finite_to_infinite} the proofs of Theorems~\ref{thmC} and~\ref{thmE}.

{\bf Acknowledgements:} We would like to thank Jim Davis, Wolfgang L\"{u}ck
and Shmuel Weinberger for helpful and stimulating conversations.


\section{Surgery preliminaries} \label{sec:surgery-preliminaries}
In this preliminary section we recall the surgery exact sequence and some related concepts from
surgery theory.  We shall assume that all manifolds have dimension $n \geq 5$.
The main references are \cite{Wall(1999)}, \cite{Ranicki(1992)} and \cite{Lueck(2001)}.

We already recalled the {\it simple structure set} $\sS^{s} (M)$ of a closed $n$-dimensional topological manifold $M$ in the introduction informally. To be more precise recall that its elements are represented by simple homotopy equivalences $f \co N \ra M$ and two such $f_{0} \co N_{0} \ra M$ and $f_{1} \co N_{1} \ra M$ are equivalent if there exists an $s$-cobordism $W$ between $N_{0}$ and $N_{1}$ together with a simple homotopy equivalence $F \co W \ra M \times [0,1]$ which restricts at the two ends to $f_{0}$ and $f_{1}$ respectively. Due to the $s$-cobordism theorem, see e.g.~\cite[Chapter 1]{Lueck(2001)}, the relation of $s$-cobordism can be replaced by that of homeomorphism.

There is also an $h$-decorated {\it structure set} $\sS^{h} (M)$ of $M$ whose elements are represented by homotopy equivalences $f \co N \ra M$ where $N$ is another closed $n$-dimensional topological manifold and two such $f_{0} \co N_{0} \ra M$ and $f_{1} \co N_{1} \ra M$ are equivalent if there exists an $h$-cobordism $W$ between $N_{0}$ and $N_{1}$ together with a homotopy equivalence $F \co W \ra M \times [0,1]$ which restricts at the two ends to $f_{0}$ and $f_{1}$ respectively. If $\pi_{1} (M) = \{ 1 \}$ then due to the $h$-cobordism theorem, see e.g.~\cite[Chapter 1]{Lueck(2001)}, the relation of the $h$-cobordism can be replaced by the condition that there exists a homeomorphism $h \co N_{0} \ra N_{1}$ such that $f_{1} \circ h \simeq f_{0}$. Due to the $s$-cobordism theorem this also holds whenever the Whitehead group $\Wh (\pi_{1} (M))$ vanishes, see e.g.~\cite[Chapter 2]{Lueck(2001)}.

In general there is a forgetful map $\sS^{s} (M) \ra \sS^{h} (M)$ which fits into the so-called
{\em Rothenberg sequence}, where the third term is $2$-torsion and depends on
various Whitehead groups from algebraic $K$-theory, see e.g.~\cite[Appendix C]{Ranicki(1992)}.

As we noted in the introduction, the main tool for calculating the simple structure set $\sS^{s} (M)$ of an $n$-dimensional manifold $M$ with $\pi = \pi_{1} (M)$ where $n \geq 5$ is the surgery exact sequence \eqref{ses:intro}
\[
\cdots \xra{~~~} L^{s}_{n+1}(\ZZ\pi) \xra{~\omega~} \sS^{s}(M) \xra{~\eta~} \sN(M) \xra{~\theta~}
L^{s}_n(\ZZ\pi).
\]
Here $\sN (M)$ are the {\it normal invariants} of $M$, which can be defined as a bordism set of degree one normal maps $(f,\overline{f}) \co N \ra M$ from $n$-dimensional manifolds $N$ to $M$. The symbol $\overline{f}$ denotes a stable bundle map $\nu_{M} \ra \xi$ which covers $f$ for some stable topological $\R^k$-bundle $\xi$ over $X$.  
The main point is that the normal invariants are calculable, they form a generalized (co-)homology theory, a fact which we will see later more concretely. The map $\eta$ is more or less obvious, it is obtained by pulling back the bundle data along the homotopy inverse of $f$.

The $L^{s}$-groups are the surgery obstruction groups defined either using quadratic forms over f.g. free based modules over the group ring $\ZZ\pi$ or using quadratic chain complexes of f.g.~free based modules over $\ZZ\pi$. The surgery obstruction map $\theta$ and the Wall realization map $\omega$ are designed to make the sequence exact.

There is also an $h$-decorated surgery exact sequence which calculates $\sS^{h} (M)$. In it, the $L^{s}$-groups are replaced by the $L^{h}$-groups defined using f.g.~free $\ZZ\pi$-modules, but the normal invariants remain the same. The Rothenberg sequences as in \cite[Appendix C]{Ranicki(1992)} quantify the difference between the $s$-decorated and $h$-decorated sequence.

In \cite{Wall(1999)} the sequence \eqref{ses:intro} is constructed as an exact sequence of pointed sets. However, it can be made into an exact sequence of abelian groups by identifying it with the algebraic surgery exact sequence of \cite{Ranicki(1992)} which is the bottom row of the following diagram:
\begin{equation} \begin{split}
	\label{eq:ses_abelian}
	\xymatrix{
	\cdots \ar[r] & L^{s}_{n+1} (\ZZ\pi) \ar[r]^{\omega} \ar[d]^{=} & \sS^{s} (M) \ar[r]^{\eta} \ar[d]^{\cong} & \sN (M) \ar[d]^{\cong} \ar[r]^{\theta} & L^{s}_{n} (\ZZ\pi) \ar[d]^{=} \\
	\cdots \ar[r] & L^{s}_{n+1} (\ZZ\pi) \ar[r] & \SS^{s}_{n+1} (M) \ar[r] & H_{n} (M;\LL_{\bullet} \langle 1 \rangle) \ar[r]_(.6){\asmb_{M}} & L^{s}_{n} (\ZZ\pi)
	}
\end{split}
\end{equation}
In the algebraic surgery exact sequence all entries are abelian groups and all maps are homomorphisms by definition. In fact all the terms are defined as $L$-groups of certain categories, but we will not need this fact. In \cite[Theorem 18.5]{Ranicki(1992)} it is shown that the vertical arrows are bijections and in this way the abelian group structures are obtained in the top row.  We will need that the term which is in bijection with the normal invariants is the homology of $M$ with respect to the $1$-connective cover $\LL_{\bullet} \langle 1 \rangle$ of the $L$-theory spectrum associated to the ring $\ZZ$. We will also need some properties of the assembly map $\asmb_{M}$. This is an important map, we note for example that the influential {\it Novikov conjecture} asserts that this map is always rationally injective. Various assembly maps including this one have been intensively studied in recent years, see for example \cite{Lueck-Reich(2005)} for an overview of the subject.

Furthermore the algebraic surgery exact sequence can be defined for any simplicial complex $K$ in the place of $M$ and it is covariantly functorial. Hence for our $M$ with the classifying $c\co M \ra B\pi$ we have a commutative diagram
\begin{equation}
	\begin{split}
\label{eqn:factoring-assembly-over-Bpi}
\xymatrix{
H_{n} (M;\LL_{\bullet} \langle 1 \rangle) \ar[r]^(.575){\asmb_{M}} \ar[d]_{c_{\ast}} & L^{s}_{n} (\ZZ\pi) \ar[d]^{=} \\
H_{n} (B\pi;\LL_{\bullet} \langle 1 \rangle) \ar[r]_(.575){\asmb_{\pi}} & L^{s}_{n} (\ZZ\pi).
}		
	\end{split}
\end{equation}

An essential point for us is the calculability of the normal invariants. Note that we have $\pi_{k} (\LL_{\bullet} \langle 1 \rangle) = \ZZ, 0, \ZZ/2, 0$ for $k \equiv 0,1,2,3$ and $k > 0$ and $\pi_{k} (\LL_{\bullet} \langle 1 \rangle) = 0$ for $k \leq 0$. Hence from the Atiyah-Hirzebruch spectral sequence we obtain that the normal invariants $\sN (M) \cong H_{n} (M; \LL_{\bullet} \langle 1 \rangle)$ form a f.g.~abelian group. In fact this homology theory can be fully understood using singular homology and topological $K$-theory via the homotopy pullback square~\cite[Remark 18.8]{Ranicki(1992)}
\[
\xymatrix{
\Omega^{\infty} \LL_{\bullet} \langle 1 \rangle \ar[rr] \ar[d] & &  \BO [1/2] \ar[d] \\
\prod_{k > 0} K(\ZZ_{(2)},4k) \times K(\ZZ/2,4k-2) \ar[rr] & &  \prod_{k > 0} K(\QQ,4k)
}
\]
However, we will mostly only use the rational information here. Note that by \cite[Remark 18.4]{Ranicki(1992)} we have the isomorphism
\begin{equation} \label{eq:HLC}
\sL \co \sN (M) \otimes \QQ \xra{\cong} H_{n} (M;\LL_{\bullet} \langle 1 \rangle) \otimes \QQ \xra{\cong} \bigoplus_{k > 0} H_{n-4k} (M;\QQ),
\end{equation}
which is given by the formula
\[
[(f,\overline{f}) \co N \ra M] \; \mapsto \; f_\ast (\sL (N) \cap [N]) - \sL (M) \cap [M],
\]
where $\sL (M) \in H^{4\ast} (M,\QQ)$ is the Hirzebruch $\sL$-polynomial in Pontryagin classes of $M$.

Let us now discuss the action of $\sAut (M)$ on $\sS^{s} (M)$ which we recalled in the introduction from the point of view of the previous remarks. Firstly note that we also have an action of $\sAut (M)$ on $\sN (M)$ also essentially given by post-composition. More explicitly, we have
\[
\sN (M) \times \sAut(M) \to \sN (M), \quad ([(f,\overline{f}) \colon N \to M], [g]) \mapsto [(g \circ f,\overline{g} \circ \overline{f}) \colon N \to M],
\]
where $\overline {f} \co \nu_{M} \ra \xi$ is the bundle map covering $f$ and the bundle map $\overline{g} \co \xi \ra (g^{-1})^{\ast} \xi$ is obtained using a homotopy inverse $g^{-1} \co M \ra M$ of $g$.


While $\sS^{s} (M)$ and $\sN (M)$ both admit abelian groups structures,
the actions above are not by automorphisms.
Indeed in \cite{Ranicki(2009)} Ranicki proves
a composition formula for these actions but we will not use Ranicki's formula here.

Next we shall discuss the versions of surgery which arise when we vary decorations.
Let $\hAut(M)$ denote the group of homotopy classes of homotopy equivalences from $M$ to $M$. It acts on the structure set $\sS^{h}(M)$ via post-composition:
\[
\sS^{h}(M) \times \hAut(M) \to \sS^{h}(M), \quad ([f \colon N \to M], [g]) \mapsto [g \circ f \colon N \to M].
\]
In this context we have the {\it $h$-decorated manifold set} $\sM^{h} (M)$ whose elements are represented by manifolds $N$ homotopy equivalent to $M$ up to $h$-cobordism. It is related to the structure set $\sS^{h}(M)$ via the forgetful map and it is easy to see that it is in fact a bijection
\[
\sS^{h}(M)/\hAut(M) \xra{\cong} \sM^{h}(M).
\]

The relations between the $4$ manifold sets which we have defined in the introduction and here are that they fit into the following commutative square
\begin{equation} \label{eqn:mfd-sets}
	\begin{split}
		\xymatrix{
			\sM^{s}(M) \ar@{>->}[d] \ar@{->>}[r] & \sM^{s/h} (M) \ar@{>->}[d]  \\
			\sM (M) \ar@{->>}[r] & \sM^{h} (M). \\
		}		
	\end{split}
\end{equation}
If the Whitehead group $\Wh (\pi)$ vanishes, which is the case for example when $\pi = \{ e \}$ or when
$\pi = \ZZ^{n}$, see \cite{Bass-Heller-Swan(1964)}, \cite[Chapter 2]{Lueck(2001)}, then all the maps are bijections but in general this may not be the case.

We conclude this preliminary section with a brief discussion of smooth surgery
and its relationship to topological surgery.
Let $M_\alpha$ be a closed smooth $n$-manifold with $n \geq 5$.
The simple smooth structure set, $\sS_{\Diff}(M_\alpha)$ and
the smooth normal invariant set $\sN_{\Diff}(M_\alpha)$ are defined
analogously to their topological counter parts, using smooth manifolds and bordisms.
Let $M$ be the topological manifold underlying $M_\alpha$.
There is a smooth surgery exact sequence for $M_\alpha$ which maps forgetfully to
topological surgery exact sequence for $M$, as shown below.
\begin{equation} \begin{split}
	\label{eq:ses_smooth}
	\xymatrix{
	\cdots \ar[r] & L^{s}_{n+1} (\ZZ\pi) \ar[r]^{\omega} \ar[d]^{=} &
	\sS^{s}_{\Diff} (M_\alpha) \ar[r]^{\eta} \ar[d]^{F_\sS} &
	\sN_{\Diff} (M_\alpha) \ar[d]^{F_\sN} \ar[r]^(0.525){\theta} & L^{s}_{n} (\ZZ\pi) \ar[d]^{=} \\
	\cdots \ar[r] & L^{s}_{n+1} (\ZZ\pi) \ar[r] &
	\sS^{s}(M) \ar[r] &
	\sN(M)  \ar[r]_(.45){\theta} & L^{s}_{n} (\ZZ\pi)
	}
\end{split}
\end{equation}
In the smooth category the normal invariants cannot be calculated via a generalised
homology theory but are computed using a generalised cohomology theory via
the bijection
\begin{equation*} 
 \sN_{\Diff}(M_\alpha) \equiv [M, \G/\OO],
\end{equation*}
where $\G/\OO$ is the homotopy fibre of the canonical map $B\OO \to B\G$
between the classifying spaces for stable vector bundles and stable spherical fibrations;
see \cite[Lemma 10.6]{Wall(1999)}.
The topological normal invariants of $M$ can be similarly computed via the bijection
$\sN(M) \equiv [M, \G/\TOP]$, where $\G/\TOP$ is the homotopy fibre of the canonical map
$B\TOP \to B\G$ and $B\TOP$ is the classifying space for stable $\R^k$-bundles;
see \cite[Theorem 10.1 Essay IV]{Kirby-Siebenmann(1977)}.
There is a canonical map $i \colon \G/\OO \to \G/\TOP$ inducing the map
$i_* \colon [M, \G/\OO] \to [M, \G/\TOP]$ which fits into the following commutative diagram:
\begin{equation} \begin{split}
	\label{eq:ses_normal_square}
	\xymatrix{
	\sN_{\Diff}(M_\alpha) \ar[d]_{F_\sN} \ar[r]^{\equiv} &
	[M, \G/\OO] \ar[d]^{i_*} \\
	\sN(M) \ar[r]_(0.4){\equiv} &
	[M, \G/\TOP]
	}
\end{split}
\end{equation}
%

\section{Determining when $|\sM(M)| = \infty$ for $\pi_{1} (M) = \{e\}$.} \label{sec:1-ctd-case}
The main subject of this section it the proof of Theorem~\ref{thmA}.
The proof of the Theorem~\ref{thmD} is similar and is presented afterwards.
We conclude the section with some examples of calculations
of $\sM(M)$.

We prefer to modify the statement of Theorem~\ref{thmA} a little and we prove the following theorem which implies the first statement of Theorem~\ref{thmA}.

We give the proof of the second statement of Theorem~\ref{thmA}
immediately after the proof
of the first statement.

\begin{thm} \label{thm:MM-infinite1}
Let $M$ be an $n$-dimensional simply connected closed manifold with $n \geq5$.  Then the following are equivalent:
\begin{enumerate}
\item
 $H^{4i}(M; \Q) \neq 0$ for some $0 < 4i < n$;
\item
$|\sS(M)| = \infty$;
\item
$|\sM(M)| = \infty$.
\end{enumerate}
\end{thm}



\begin{proof}[Proof of (1) $\Leftrightarrow$ (2)]
Since $M$ is simply connected the surgery exact sequence for $M$ simplifies to the short exact sequence
\[
\xymatrix{
0 \ar[r] & \sS (M) \ar[r]^{\eta} & \sN (M) \ar[r]^{\theta} & L_{n} (\ZZ) \ar[r] & 0,
}
\]
see e.g.~\cite[Chapter 20]{Ranicki(1992)}. Tensoring with $\QQ$ yields the short exact sequence
\[
\xymatrix{
0 \ar[r] & \sS (M) \otimes \QQ \ar^{\eta}[r] & \sN (M) \otimes \QQ
\ar^{\theta}[r] & L_n (\ZZ) \otimes \QQ \ar[r] & 0.
}
\]
As noted earlier we have the isomorphism
\[
\sL \co \sN (M) \otimes \QQ \ra \bigoplus_{i > 0} H_{n-4i} (M;\QQ)
\]
and in addition we know that in the case $n = 4m$ the surgery obstruction map $\theta$ sends $H_{0} (M;\QQ) \cong \QQ$ isomorphically onto $L_n (\ZZ) \otimes \QQ \cong \QQ$ \cite[Chapter 20]{Ranicki(1992)}. Hence $\sS (M)$ is infinite if and only if $\sS (M) \otimes \QQ \neq 0$ if and only if its image under $\sL \circ \eta$ which equals $\oplus_{0 < 4i < n} H_{n-4i} (M;\QQ)$ is non-zero. The Poincar\'e duality isomorphism $- \cap [M] \co H^{4i} (M;\QQ) \cong H_{n-4i} (M;\QQ)$ now yields the desired statement.	
\end{proof}

\begin{proof}[Proof of (3) $\Rightarrow$ (2)] Clear. \end{proof}

For the proof of (2) $\Rightarrow$ (3) we need some preparation and therefore the proof itself is
presented at the end of this section. In Definition~\ref{def:divi} we shall construct for $0 < 4k < n$ such that $H^{4k} (M;\QQ) \neq 0$ a certain map, denoted
 \[
 \divi_{k} \co \sS (M) \ra \NN,
 \]
with the property that the value of $\divi_{k} (f)$ for a manifold structure $f \co N \ra M$ only depends on $N$ and not on
$f$ and then to show that the image of $\divi_{k}$ is infinite in $\ZZ$.
This clearly is enough.

Let $0 < 4k < n$ be such that $H^{4k} (M;\QQ) \neq 0$ and consider the
function
\[
\sL_k \co \sS (M) \ra H^{4k} (M;\QQ) \quad
[f \co N \ra M ] \mapsto (f^{-1})^\ast \sL_k (N) - \sL_k (M).
\]
The values of $\sL_k$ depend on $f$ and so it is not a good
candidate for $\divi_k$. Nevertheless it is a first step on the way
towards it. We note that the function $\sL_k$ is a homomorphism of abelian groups.
This follows from the identity
\begin{equation} \label{eq:sL}
(- \cap [M]) \circ \sL_k = \pr_k \circ \sL \circ \eta,
\end{equation}
where $\pr_k$ is the projection from $\bigoplus_{i>0} H_{n-4i}(M; \Q)$ to
$H_{n-4k}(M; \Q)$ and the fact that
$(-\cap [M]), \pr_k, \sL$ and $\eta$ are all homomorphisms.
Equation~\eqref{eq:sL} in turn follows from the appropriate projection of the equation
\[
((f^{-1})^\ast \sL (N) - \sL (M)) \cap [M] = f_\ast (\sL (N) \cap [N]) - \sL (M) \cap [M].
\]

The Hirzebruch $\sL$-class $\sL_k (-) \in H^{4k} (-;\QQ)$ is given as a rational linear combination of products of
Pontryagin classes $p_i$. Let $c_k$ be the lowest common multiple of the absolute values of the denominators of the coefficients in the expression
defining $\sL_k (-)$. For example
\[
\sL_{2} = \frac{7}{45} p_{2} - \frac{1}{45} p_{1}^{2}
\]
and so $c_{2} = 45$. For a space $X$ let $FH^{i} (X;\ZZ) := H^{i} (X;\ZZ)/\textup{tors}$ denote the free part of the cohomology. Then for smooth manifolds $M_\alpha$ the term $\sL_k (M_\alpha)$ belongs to the lattice $(1/c_k) \cdot FH^{4k} (M_\alpha;\ZZ)$ in the $\QQ$-vector space $H^{4k} (M_\alpha;\QQ)$, since for smooth manifolds the Pontryagin classes $p_{k} (M) \in  H^{4k} (M;\ZZ)$ are integral.  For topological manifolds we have rational Pontryagin classes $p_{k} (M) \in  H^{4k} (M;\QQ)$
and a rational equivalence $\BSO \ra \BSTOP$ see \cite[Annex 3, section 10]{Kirby-Siebenmann(1977)}, \cite{Novikov(1966)}.
It follows that for each $k > 0$ that there are universal positive integers $t_{k} > 0$
such that
\[
p_k(M) \in (1/t_k) \cdot FH^{4k} (M;\ZZ).
\]
%
Setting $t$ to be the lowest common multiple of the $t_i$ where $0 < 4i \leq n$
we see that $\sL_k (M)$ belongs to the lattice
$$(1/(c_k \cdot t)) \cdot FH^{4k} (M;\ZZ) \subset H^{4k} (M;\QQ),$$
where by a {\em lattice} $L$ we simply mean a finitely generated free abelian group
$L \cong \Z^r$. A {\em sublattice} of $L$ is a subgroup $L' \subset L$; it is called {\em full} if
$L' \subset L$ has finite index. For each $k$ set
\begin{equation} \label{eqn:defn-of-r-k}
r_k :=  c_k \cdot t.
\end{equation}

\noindent
{\bf Remark:}
In the first version of this paper posted to the arXiv we made a mistake
by omitting the factor $t$ in the definition of $r_k$.

\begin{lem} \label{lem:image-of-str-set-is-lattice}
The image $\sL_k (\sS (M))$ is a non-zero full sublattice of
\[
(1/r_k) \cdot FH^{4k} (M;\ZZ) \subset H^{4k} (M;\QQ).
\]
\end{lem}



\begin{proof}
We already have that $\sL_k$ is a homomorphism of abelian groups. As the image of finitely generated abelian group
$\sL_k (\sS (M))$ forms a lattice in $H^{4k} (M;\QQ)$ whose values belong to $(1/r_k) \cdot FH^{4k} (M;\ZZ)$. That it is a non-zero lattice follows from the isomorphism
\[
\sL \circ \eta \co \sS (M)_{\QQ} := \sS (M) \otimes \QQ \xra{\cong} \bigoplus_{0 < 4i < n}
H_{n-4i} (M;\QQ).
\]
We next show that it is a full sublattice.
Before rationalising we still have the homomorphism (which might not be an isomorphism anymore)
\[
\sL \co \sN (M)/\textup{tors} \ra \bigoplus_{i > 0} (1/r_{i}) \cdot FH_{n-4i} (M;\ZZ)
\]
given by the same formula
\[
[(f,\overline{f}) \co N \ra M] \; \mapsto \; f_\ast (\sL (N) \cap [N]) - \sL (M) \cap [M].
\]
Denote
\[
H_M := \bigoplus_{i > 0} (1/r_{i}) \cdot FH_{n-4i} (M;\ZZ)
\]
and
\[
(H_M)_{\QQ} := \bigoplus_{i > 0} H_{n-4i} (M;\QQ).
\]
Consider the commutative diagram
\[
\xymatrix{
\sS (M) \ar[r]^(0.55){\sL \circ \eta} \ar[d] & H_M \ar[d] \\
\sS (M)_{\QQ} \ar[r]_(0.5){\sL \circ \eta} & (H_M)_{\QQ}
}
\]
Since $\sL \circ \eta \co \sS (M)_{\QQ} \ra (H_M)_{\QQ}$ is onto, it follows that $\sL \circ \eta \co \sS (M) \ra H_M$ has image a full sublattice of $H_M$. Projecting to the summand $(1/r_k) \cdot FH_{n-4k} (M;\ZZ) \subset H^{4k} (M;\QQ)$ we obtain a full sublattice of $(1/r_k) \cdot FH_{n-4k} (M;\ZZ)$ and the lemma follows by the Poincar\'e duality isomorphism $(1/r_k) \cdot FH_{n-4k} (M;\ZZ) \cong (1/r_k) \cdot FH^{4k} (M;\ZZ)$.
\end{proof}

In order to obtain the function $\divi_k$ from $\sL_k$ we
recall the concept of divisibility for elements of lattices and establish some elementary properties of divisibility.
Recall that a lattice is simply a finitely generated free abelian group.

\begin{defn} \label{defn:divisibility-in-lattice}
Let $L$ be lattice. For $x \in L-\{0\}$ define the
{\em divisibility of $x$}, denoted by $\divi (x) = d \in \NN$, to be the largest
number such that $x = d \cdot x_0$ for some $x_0 \in L$.
We set $\divi(0) := 0$.
\end{defn}

\begin{lem} \label{lem:divi-vs-iso}
If $h \co L_0 \ra L_1$ is an isomorphism of lattices then for all $x \in L_0$ we have
$\divi (x) = \divi (h(x))$.
\end{lem}

\begin{lem} \label{lem:divi-vs-lattice}
If $L \neq \{0\}$ is a non-zero lattice, then the set
$\{ \divi (x) \, | \, x \in L \}$ is infinite.
\end{lem}

We define an {\em affine sublattice} of $L' \subset L$ to be a coset
$l_0 + L_0 \subset L$
where $L_0 \subset L$ is a lattice;
it is called {\em full} if $L_0$ is full.
We define $\divi(L') := \{\divi(l) \,|\, l \in L' \} \subseteq \NN$ to be the set of
divisibilities of elements in $L'$.

\begin{lem} \label{lem:divi-vs-affine-lattice}
If $L \neq \{0\}$ is a non-zero lattice and $L' \subseteq L$ is a full affine sublattice,
then $|\divi(L')| = \infty$.
\end{lem}

\begin{proof}
Let $L' = l_0 + L_0$ for a full sublattice $L_0$,
set $T : = L/L_0$ and consider the short exact sequence
$$ 0 \to L_0 \to L \to T \to 0.$$
Since $T$ is finite there are infinitely many primes $p$ which are prime to $|T|$.
Let $p$ be such a prime.
Since $T \otimes \Z/p = 0$, when we
tensor the sequence above with $\Z/p$ we obtain
an exact sequence
$$ L_0 \otimes \Z/p \to L \otimes \Z/p \to 0.$$
Hence we may find $l_p \in L_0$ such that $l_0 \equiv l_p$~mod~$p$.
It follows that $l_0 - l_p \in L'$ is divisible by $p$
and so $p \in \divi(L')$.  As there are infinitely many primes
prime to the order of $T$ we conclude that
$|\divi(L')| = \infty$.
\end{proof}

For a manifold $N$ we define $L_N := (1 / r_k) \cdot FH^{4k} (N)$ where $r_{k}$ is defined by the equation~\eqref{eqn:defn-of-r-k}. We have $\sL_k (N) \in L_N$.

\begin{defn} \label{def:divi}
For a manifold structure $f \colon N \to M$, define $\divi_k (f) = \divi (\sL_k (N))$.
\end{defn}

\begin{proof}[Proof of (2) $\Rightarrow$ (3)]
A connection between $\sL_k$ and $\divi_k$ will be obtained via
Lemma \ref{lem:divi-vs-affine-lattice}. Recall that for $f \co N \ra M$
representing an element in $\sS (M)$ we have the isomorphism
$(f^{-1})^\ast \co (1 / r_k) \cdot FH^{4k} (N) \ra (1 / r_k) \cdot FH^{4k} (M)$ and it follows from
Lemma \ref{lem:divi-vs-iso} that
\[
\divi (\sL_k (N)) = \divi ((f^{-1})^\ast \sL_k (N)).
\]
Moreover, we can consider the full affine sublattice
\[
L'_{M} := \sL_{k} (M) + \sL_{k} (\sS (M)) = \{ (f^{-1})^\ast (\sL_k (N)) \} \subset L_{M}.
\]
By Lemma \ref{lem:image-of-str-set-is-lattice} $L'_{M} \subset L_{M}$ is a full sublattice and so $|\divi (L'_{M})| = \infty$ by Lemma~\ref{lem:divi-vs-affine-lattice}. It follows immediately that $|\divi_{k} (\sS (M))| = \infty$ and so $|\sM (M) = \infty|$.
\end{proof}

\begin{proof}[Proof of the second statement in Theorem~\ref{thmA}]
	In the preceding proof of the implication (2) $\Rightarrow$ (3) of Theorem~\ref{thm:MM-infinite1} we showed that $|\sS(M)| = \infty$ implies $|\divi_{k} (\sS(M))| = \infty$ which together with the implication (3) $\Rightarrow$ (2) gives the desired statement.
\end{proof}

\begin{proof}[Proof of Theorem \ref{thmD}]
Consider the forgetful map $F_\sS \colon \sS_{\Diff}(M_\alpha) \to \sS(M)$ for
a closed smooth simply-connected $n$-manifold $M_\alpha$ with $n \geq 5$
and let $\TOP/\OO$ be the fibre of the canonical map
$B\OO \to B\TOP$.
By smoothing theory, see \cite[p.\,289, Appendix C Essay V]{Kirby-Siebenmann(1977)},
the group of homotopy classes $[M, \TOP/\OO]$ acts transitively on
the non-empty pre-images of $F_\sS$.
By \cite[Theorem 5.5 (II), Essay V]{Kirby-Siebenmann(1977)} the groups
$\pi_i(\TOP/\OO)$ are finite for all $i$ and so $[M, \TOP/\OO]$ is finite
and the map $F_\sS$ is finite to one.
Thus if $\sS(M)$ is finite then $\sS_{\Diff}(M_\alpha)$ is also finite.

Conversely, suppose that $\sS(M)$ is infinite.
By Theorem \ref{thm:MM-infinite1} and its proof we know for some $k$
with $0 < 4k < n$ that
$(1/r_k) \cdot FH_{n-4k}(M; \Z)$ is non-zero and contains
$\sL_k \circ \eta(\sS(M))$ as a full affine sub-lattice.
By a theorem of Weinberger \cite{Weinberger(1990)},
$F_\sS(\sS_{\Diff}(M_\alpha)) \subset \sS(M)$ contains a subgroup
of finite index.
Hence $\sL_k \circ \eta(\sS(M)) \circ F_\sS$ is also a full affine sublattice
of $(1/r_k) \cdot FH_{n-4k}(M; \Z)$ and so by Lemma \ref{lem:divi-vs-affine-lattice}
the set $\divi_k(\sS_{\Diff}(M_\alpha))$ is infinite.
This completes the proof.
\end{proof}

We conclude this section with some examples of calculations of
$\sM(M)$.


\begin{expl}[The manifold set of $\C P^n$]
For $n \geq 3$, the surgery exact sequence for $\C P^n$ gives an isomorphism
$$ \sS(\C P^n) \cong \bigoplus_{i > 0} H_{2n-4i-2}(\C P^n; \Z/2)
\bigoplus_{i > 0} H_{2n-4i}(\C P; \Z).$$
The action of $\sAut(\C P^n) = \{\pm 1\}$ on $\sS(\C P^n)$ is identified with
the sign action on this groups and hence there is a bijection
$$ \sM(\C P^n) \cong \left( \bigoplus_{i > 0} H_{2n-4i-2}(\C P^n; \Z/2)
\bigoplus_{i > 0} H_{2n-4i}(\C P; \Z) \right)/\{\pm 1\}.$$
\end{expl}

\begin{expl}[The manifold set of $W_g = \sharp_g(S^{4k} \times S^{4k})$]
\label{expl:Wg}
For $k \geq 1$, define the manifold $W_g := \sharp_g(S^{4k} \times S^{4k})$
and assume $g \geq 1$.
Computing the surgery exact sequence of $W_g$
gives an isomorphism
$$ \sS(W_g) \cong
H_{4k}(W_g; \Z).$$
The intersection form of $W_g$ is canonically
identified with $H_{+}(\Z^g)$, the standard hyperbolic form on $\Z^g$
and we identify
$$\sS(W_g) = H_{+}(\Z^g).$$
With this identification, the action of $\sAut(W_g)$ on
$\sS(W_g)$ factors over the action of
the quasi-orthogonal group, $\OO(g) \rtimes \Z/2$, on the set $H_{+}(\Z^g)$.
Indeed, there is a bijection
$$ \sM(W_g) \equiv
H_{+}(\Z^g) / (\OO(g) \rtimes \Z/2)$$
and by Theorem \ref{thmA} the map $\divi_k \colon \sM(W_g) \to \NN$ has infinite image.
It is not hard to show that $\wt \sS(W_g)$,
the reduced structure group of $W_g$ whose definition is recalled in Section \ref{ss:further},
satisfies
$\wt \sS(W_g) \otimes \Q = 0$.
Hence $|\wt \sS(W_g)| < \infty$,
whereas $|\sM(W_g)| = \infty$.
\end{expl}

\section{Conditions ensuring $|\sM^{s/h}(M)| = \infty$ when $\pi_{1} (M) \neq \{ e \}$.} \label{sec:non-1-ctd-case}
In this section we prove Theorem \ref{thmB}.
Comparing the cardinalities of $\sS^s(M)$ and $\sM^{s/h}(M)$ is more complicated
when $M$ is not simply-connected.
Most importantly, for an $n$-dimensional manifold $M$ with $\pi = \pi_{1} (M)$ and $n \geq 5$ the surgery exact sequence~\eqref{ses:intro} does not split into short exact sequences in general.
The $L$-group $L^{s}_{n+1} (\ZZ \pi)$ may be infinite and the map $\omega$ may be injective on an infinite subgroup of $L^{s}_{n+1} (\ZZ \pi)$.  In such a case we can have $\sS^{s} (M)$ infinite and $\sM^{s} (M)$ finite as is shown in the next section.

However, we would still like to find conditions which guarantee that $\sM^{s}_{h} (M)$ is infinite
(hence $\sM^{s} (M)$, $\sM (M)$, and $\sM^{h} (M)$ as well, see \eqref{eqn:mfd-sets})
when $\eta(\sS^s(M))$ is infinite. In this case in Theorem~\ref{thmB} we offer three alternative additional assumptions each of which implies that $\sM^{s}_{h} (M)$ is infinite when $\eta(\sS^s(M))$ if infinite.
We do not know whether they are also necessary.

A word about decorations is needed at this point. On one hand in the assumptions of Theorem we have $|\eta (\sS^{h} (M))| = \infty$, on the other hand the conclusion is about the $s/h$-decorated manifold set
$\sM^{s}_{h} (M)$. We note firstly that the invariant $\divi_{k} (f)$ is clearly an $h$-cobordism invariant,
secondly that we have the commutative diagram
\[
\xymatrix{
\sS^{s} (M) \ar[r]^{\eta} \ar[d] & \sN (M) \ar[d]^{=} \\
\sS^{h} (M) \ar[r]_{\eta}  & \sN (M)
}
\]
and thirdly the existence of the Rothenberg sequences of \cite[Appendix C]{Ranicki(1992)}. The relative terms in the Rothenberg sequences are $2$-torsion and $\sN (M)$ is a f.g. abelian group. It follows that $\eta(\sS^{s} (M))$ is a finite index subgroup of $\eta(\sS^{h} (M))$.
In the following we will therefore
talk about $\eta$ and $\im (\eta)$ and will mean either of the two possible maps depending on the context.

The proofs of the respective items of Theorem~\ref{thmB} are similar to the proof of the implication (1) $\Rightarrow$ (3) of Theorem~\ref{thm:MM-infinite1}. The main difficulties which force us to add further assumption are due to the fact that the surgery exact sequence is more complicated in the non-simply-connected case. For example even if we know that $|\im (\eta)|$ is infinite we were not able to find a lattice $L \subset \sN (M) \otimes \QQ$ and a non-zero lattice $L' \subset \im (\eta)/\textup{tors}$ such that for some $0 < 4k <n$ the coset $\sL_{k} (M) + L'$ would be an affine full sublattice in $L$ in general.

\begin{proof}[Proof of Theorem~\ref{thmB} (1)]
The surgery exact sequence is functorial for homotopy equivalences
$f \colon M \to M'$.  In particular $f$ induces a bijection $\sS^h(M) \to \sS^h(M')$.
Hence we may replace $M$ by $M'$.
In this case we have that $\sL_{k} (M') = 0$ for $k > 0$.
Hence we obtain that $\im (\sL \circ \eta)/\textup{tors}$ is a non-zero lattice in $H$ as above and the divisibilities of its elements
$\sL \circ \eta ([f \co N \ra M'])$ do not depend on $f$.
\end{proof}

\begin{proof}[Proof of Theorem~\ref{thmB} (2)]
In this case we have that $\sL_{k} (M) + L$ is an affine full sublattice of $L'$ and for the divisibilities of its elements $(f^{-1})^{\ast} \sL_{k} (N)$ with respect to
the lattice $(1/r_{k}) \cdot FH^{4k} (M; \ZZ)$ we have that they do not depend on $f$. Since the lattice $L'$ is non-zero it must contain an infinite number of divisibilities as must any of its affine full sublattices.
\end{proof}

There are various ways to formulate Theorem~\ref{thmB} (3). Note that we want to obtain that the manifold set is infinite as a corollary of the fact that $\im (\eta) = \ker (\theta)$ is infinite and an additional assumption. However in practice it is difficult to decide when $(f,\overline{f}) \co N \ra M$ belongs to $\ker (\theta)$. On the other hand we have that the surgery obstruction map factors through the assembly map for the fundamental group $\pi$ via the homomorphism induced by the classifying map $c \co M \ra B\pi$. Hence we have $\ker (c_{\ast}) \subset \ker (\theta) = \im (\eta)$. Therefore in the following theorem we replace the condition $|\im (\eta)| = \infty$ by $|\ker (c_{\ast})| = \infty$. Alternatively we can assume that $|\im (\eta)| = \infty$ and $\pi$ satisfies the Novikov conjecture. Theorem~\ref{thmB} (3) follows from the following.

\begin{thm} \label{thm:MM-infinite-non-1-ctd-v4}
 If $|\ker (c_{\ast})| = \infty$, the fundamental group $\pi$ is a Poincar\'e duality group of dimension $n$ and the classifying map $c \co M \ra B\pi$ has non-zero degree, then we have that $|\sM^{s/h}(M)| = \infty$.
\end{thm}

\begin{proof}
 We use the following commutative diagram
 \[
\xymatrix{
 \bigoplus_{0 < 4k < n} FH^{4k} (M; \Z) \ar[d]_{- \,\cap \, [M]} & \bigoplus_{0 < 4k < n}
 FH^{4k} (B\pi; \Z)  \ar[l]_{c^{\ast}} \ar[d]^{- \,\cap \, \deg(c_{\ast}) \cdot [B\pi]} \\
 \bigoplus_{0 < 4k < n} FH_{n-4k} (M; \Z) \ar[r]_{c_{\ast}} & \bigoplus_{0 < 4k < n}
 FH_{n-4k} (B\pi; \Z)
}
 \]
The left vertical arrow is an isomorphism, the right vertical arrow is an isomorphism if $\deg (c_{\ast}) = 1$ otherwise it is a composition of an isomorphism with a multiplication by $\deg (c_{\ast}) \neq 0$ and hence injective with a finite cokernel.

The diagram induces a map $\ker (c_{\ast}) \ra F \coker (c^{\ast})$, where $F \coker (c^{\ast})$ denotes the free part of the finitely generated abelian group $\coker (c^{\ast})$, which is an isomorphism if $\deg (c_{\ast}) = 1$ and if more generally $\deg (c_{\ast}) \neq 0$ then the image is a full sublattice. Consider the projection
 \[
 \textup{proj} \co \bigoplus_{0 < 4k < n} (1/r_{k}) \cdot FH^{4k} (M ; \ZZ) \ra (1/r_{k}) \cdot F \coker (c^{\ast}).
 \]
Next we note that the map $c^{\ast}$ is equivariant with respect to the action of the homotopy automorphisms. Therefore we can take the divisibilities after the projection (that means in the lattice $(1/r_{k}) \cdot F \coker (c^{\ast})$)
 \[
\divi \; \textup{proj} (f^{-1})^{\ast} \sL (N) = \divi \; (f^{-1})^{\ast} \textup{proj} \sL (N)
 \]
which again do not depend on $f$. By previous discussion the elements from $\ker (c_{\ast})$ form an affine full sublattice and hence we have an infinite number of their divisibilities.
\end{proof}

\section{Examples where $|\sS(M)| = \infty$ but $|\sM_\pi(M)| = 1$.}
\label{sec:finite_to_infinite}
%
In this section we prove Theorems \ref{thmC} and \ref{thmE}.
The manifolds $M_{r, g}$ appearing in Theorem \ref{thmC} were defined in the introduction
and we begin by repeating their definition with some extra details.

Let $T^r = S^1 \times \dots \times S^1$ be the $r$-torus
for a positive integer $r$.
For $r = 1$, we regard $S^1$ a $CW$-complex with one $0$-cell and one $1$.
In general we regard $T^r$ as the product $CW$-complex and we let
\[ K := (T^r)^{(2)} \]
be the $2$-skeleton of $T^r$.
The complex $K$ can also be defined as the union of all the co-ordinate $2$-tori in $T^r$.
Let $k \geq 1$ be an integer.
Since $K$ is $2$-dimensional it
can be embedded in $\R^{4k+3}$
and we let $W_K \supset K$ be a regular neighbourhood.
We note that the inclusion $K \hookrightarrow W_K$ is what Wall calls the {\em trivial thickening}
of $K$ \cite[\S 3]{Wall(1966)}.
We define the closed (smooth, orientable) $(4k{+}2)$-manifold
\[  M_0 := \del W_K  \]
to be the boundary of $W$.
For an integer $g \geq 0$ we have the closed smooth oriented $(4k{+}2)$-manifold
\[ M_{r,g} := M_0 \sharp_g (S^{2k+1} \times S^{2k+1}) .\]
%

By construction the fundamental group of $M_{r, g}$ is identified with $\Z^r$ by
the following isomorphisms
\[ \pi_1(M_{r,g}) \cong \pi_1(M_{r, 0}) \cong \pi_1(W_K) \cong  \pi_1(K) \cong \pi_1(T^r) \cong \Z^r.\]
%
%
%
Let $\Z[\Z^r]$ denote the (untwisted) group ring of the group $\Z^r$
and consider the following fragment of the surgery exact sequence for $M_{r,g}$,
\begin{equation} \label{eq:lses}
\sN (M_{r,g} \times I) \xra{\theta_{4k+3}}
L_{4k+3}(\ZZ[\Z^r]) \xra{~\omega~}
\sS (M_{r,g}) \xra{~\eta~}
\sN (M_{r,g})  \xra{\theta_{4k+2}}
L_{4k+2} (\ZZ[\Z^r]),
\end{equation}
where we have added subscripts to the surgery obstruction maps
$\theta$ which indicate the dimension of the relevant manifolds.
The proof of Theorem \ref{thmC} follows quickly from the following three lemmas.
\begin{lem} \label{lem:imL}
If $r \geq 3$ the image of $\omega \colon L_{4k+3}(\ZZ[\Z^r]) \to \sS(M_{r,g})$ is infinite.
\end{lem}

\begin{lem} \label{lem:saut}
We have
$\eta \bigl( \sS(M_{r,g}) \bigr)
= \eta \bigl([\mathrm{Id} \colon M_{r, g} \to M_{r, g}] \sAut_\pi(M_{r,g})\bigr)$.
\end{lem}

\begin{lem} \label{lem:triv_actn}
If $ g \geq r{+}3$
the orbits of the action of $L_{4k+3}(\Z[\Z^r])$ on $\sS(M_{r,g})$
map trivially to $\sM^s_\pi(M)$.
\end{lem}

The proofs of these three lemmas require some
preliminary results.
The complex $K$ is a $2$-dimensional CW complex with $r$ $1$-cells,
and $s := C^r_2$ $2$-cells.
We let $i \colon K \to T^r$ be the inclusion and
leave the reader to check the following simple lemma.
\begin{lem} \label{lem:(co)homology}
For all $j \leq 2$ the induced homomorphisms
$i_* \colon H_j(K; \Z) \to H_j(T^r; \Z)$ and $i^* \colon H^j(T^r; \Z) \to H^j(K; \Z)$
and isomorphisms.
\end{lem}
The next lemma describes the homotopy type and stable homotopy type of $M_{r,g}$.
\begin{lem} \label{lem:M_{r,g}}
There is a homotopy equivalence
\[ M_{r,g} \simeq
\bigl (K \vee (\vee_{2g} S^{2k+1}) \vee (\vee_{s} S^{4k}) \cup (\cup_{r} D^{4k+1}) \bigr)
\cup D^{4k+2}\]
and a stable homotopy equivalence
\[ M_{r,g} \sim
(\vee_r S^1) \vee (\vee_s S^2) \vee (\vee_{2g} S^{2k+1}) \vee (\vee_{s} S^{4k})
\vee (\vee_r S^{4k+1})
\vee S^{4k+2}.\]
\end{lem}

\begin{proof}
Since $4k{+}3 > 2\dim(K) = 4$, there is a diffeomorphism $W_K \cong V_K \times [0, 1]$,
where $V_K$ is a regular neighbourhood of an embedding $K \hookrightarrow \R^{4k+2}$:
in the language of Wall, $K \hookrightarrow W_K$ is a {\em stable thickening}, see
\cite[\S5]{Wall(1966)}.  It follows that $M_0$ is diffeomorphic to the trivial double of $V_K$:
$$ M_{r, 0} \cong V_K \cup_{\id} (-V_K)$$
Since $V$ is a stable thickening of $K$, the proof of \cite[Embedding Theorem, \S2]{Wall(1966)}
shows that $V$ has a handle decomposition with one $j$-handle for every $j$-cell of $K$.
Hence $M_{0, r}$ has a handle decomposition with one $j$-handle for every $j$-cell of $K$
and one $(4k{+}2-{j})$-handle very every $j$-handle of $K$.
This gives the required cell structure for $M_{r, 0}$ and it remains to show that the $4k$-handles
are attached trivially.
To see this, we note that $V_K$ and $W_K$ are canonically homotopy equivalent to $K$ and so the maps
$V_K \to M_{r, 0} \to W_K$ define a homotopy retraction $K \to M_{r, 0} \to K$.
Since $K$ is the $2$-skeleton of $M_{r, 0}$ and it is a retract it follows that the $4k$-cells
are attached trivially.
This proves the first part of the lemma for $M_{r, 0}$.
The first part of the lemma for $M_{r, g}$ follows immediately from the standard cell decomposition
of $\sharp_g(S^{2k+1} \times S^{2k+1})$ and properties of the connected sum operation.

The stable homotopy type of $M_{r, g}$ is clearly that of $M_{r, 0}$ wedge $2g$ copies of
$S^{2k+1}$.  To determine the stable homotopy type of $M_{r, 0}$ we consider the co-fibration sequence
$$ K \to M_{r, 0} \to M_{r, 0}/K.$$
Since $K \to M_{r, 0}$ admits a retraction, standard arguments show that $M_{r, 0}$ has the
stable homotopy type of the wedge of $K$ and $M_{r, 0}/K$:
$$ M_{r, 0} \sim K \vee (M_{r, 0}/K).$$
As $K$ is the $2$-skeleton of $T^r$ there is a stable equivalence
$K \sim (\vee_r S^1) \vee (\vee_s  S^2)$.
Moreover, there is a homotopy equivalence $M_{r, 0}/K \simeq V_K/\del V_K$
and since $V_K$ is a regular neighbourhood of an embedding $K \hookrightarrow \R^{4k+2}$
it follows that $V/\del V$ is a model for the $S$-dual of $K_{+}$, the space $K$ with an
additional disjoint base-point.  The proof is analogous to the proof of Milnor-Spanier for manifolds
\cite{Milnor-Spanier(1960)}.
It follows that there is a stable equivalence
$$V_K/\del V_K \sim (\vee_{s} S^{4k})
\vee (\vee_r S^{4k+1}) \vee S^{4k+2}$$
and the lemma follows.
\end{proof}

Next we recall Shaneson's computation of the relevant $L$-groups of $\Z[\Z^r]$.

\begin{lem}[C.f.~{\cite[Theorem 5.1]{Shaneson(1969)}}]
\label{lem:shaneson}
Set $u: = C^r_3$.  There are isomorphisms
$$ L_{4k+2}(\Z[\Z^r]) \cong H_2(T^r; \Z) \oplus H_0(T^r; \Z/2) \cong \Z^s \oplus \Z/2 $$
and
$$
L_{4k+3}(\Z[\Z^r]) \cong H_3(T^r; \Z) \oplus H_1(T^r; \Z/2) \cong \Z^u \oplus (\Z/2)^r.$$
%
\end{lem}

The following lemma proves Lemma \ref{lem:imL}.

\begin{lem} \label{lem:ses}
The surgery exact sequence for $M_{r,g}$ reduces to a short exact sequence
$$ 0 \to \Z^u \to \sS(M_{r,g}) \to (\Z/2)^s \to 0,$$
where $\Z^u \cong \im(\omega)$ and $(\Z/2)^s \cong \ker(\eta)$.
\end{lem}

\begin{proof}
Let $c \colon M_{r,g} \to B(\Z^r) = T^r$ be the classifying map for the
universal covering of $M_{r,g}$ and let $A_1 \colon H_*(T^r; \Lone) \to L_*(\Z[\Z^r])$
be the indicated assembly homomorphism.
As we saw in \eqref{eq:ses_abelian} of Section \ref{sec:surgery-preliminaries},
the surgery exact sequence \eqref{eq:lses} can be identified as long exact sequence
of abelian groups where
$$ \sN(M_{r,g}) \cong H_{4k+2}(M_{r,g}; \Lone),
\quad
\sN(M_{r,g} \times I) \cong H_{4k+3}(M_{r,g}; \Lone)
$$
and the surgery obstruction homomorphisms $\theta_*$ factor as
$\theta_* = A_* \circ c_*$.

We compute $H_{m}(M_{r,g}; \Lone)$ via the Atiyah-Hirzebruch spectral sequence
$$ \bigoplus_{i=0}^{m-1}H_i(M_{r,g}; L_{m-i}(\Z)) \Longrightarrow H_{m}(M_{r,g}; \Lone),$$
which collapses since $M_{r,g}$ has the stable homotopy type of a wedge of spheres
by Lemma \ref{lem:M_{r,g}}.
For $m = 4k{+}2$ we compute that
$$ H_{4k+2}(M_{r,g}; \Lone) \cong
H_2(M_{r,g}; L_4(\Z)) \oplus
H_4(M_{r,g}; L_2(\Z)),$$
$c_*$ maps the first summand isomorphically onto $H_2(T^r; L_4(\Z))$
(using Lemma \ref{lem:(co)homology})
and $c_*$ vanishes on the second summand.
Since $A_1 \colon H_*(T^r; \Lone) \to L_*(\Z[\Z^r])$ is injective
\cite[Example 24.16]{Ranicki(1992)},
it follows that
$$ \ker(\theta_{4k+2}) \cong H_4(M_{r,g}; L_2(\Z)) \cong (\Z/2)^s.$$
For $m = 4k{+}3$ we have that
$$ H_{4k+3}(M_{r,g}; \Lone) \cong
H_1(M_{r,g}; L_6(\Z)) \oplus
H_3(M_{r,g}; L_4(\Z)) \oplus
H_5(M_{r,g}; L_2(\Z),$$
$c_*$ maps the first summand isomorphically onto $H_1(T^r; L_6(\Z))$
(using Lemma \ref{lem:(co)homology})
and $c_*$ vanishes on the second and third summands.
Since
$A_1 \colon H_*(T^r; \Lone) \to L_*(\Z[\Z^r])$ is injective
it follows that
$$ \coker(\theta_{4k+3}) = \im(\omega) \cong H_3(T^r; L_4(\Z)) \cong \Z^u.$$
The completes the proof of the lemma.
\end{proof}

\begin{proof}[Proof of Lemma \ref{lem:saut}]
By Lemma \ref{lem:ses} and its proof we have isomorphisms
$$\eta(\sS(M_{r,g})) \cong H_4(M_{r,g}; L_2(\Z)) = H_4(M_{r,g}; \Z/2) \cong (\Z/2)^s. $$
Given $x \in H_4(M_{r,g}, \Z/2)$,
we must find a self-homotopy equivalence
$g_x \colon M_{r,g} \to M_{r,g}$ whose normal invariant is $x$.
Such homotopy equivalences are well-known and are provided by {\em pinch maps}.
For a description of pinch maps see \cite[\S 4]{Madsen-Taylor-Williams(1980)}.

By Lemma \ref{lem:M_{r,g}}(a) there is an embedding
$h_x \colon S^{4k} \to M_{r,g}$ with trivial normal bundle which represents $x$.
We let $\eta^2_{4k} \colon S^{4k+2} \to S^{4k}$ be essential and define $g_x$
to be the composition
$$ M_{r,g} \xra{~c~} M_{r,g} \vee S^{4k+2}
\xra{~\mathrm{Id} \vee \eta^2_{4k}~} M_{r,g} \vee S^{4k}
\xra{~\mathrm{Id} \vee h_x~} M_{r,g},
$$
where $c$ is the map collapsing the boundary of a small embedded $D^{4k+2}$-disc
embedded in $M_{r,g}$.
Using \cite[Lemma 7.4]{Crowley&Hambleton(2015)}, standard
arguments show that $\eta(g_x) = x$.
Since the element $x \in H_4(M_{r,g}; \Z/2)$ was arbitrary, this proves Lemma \ref{lem:saut}.
\end{proof}

\begin{proof}[Proof of Lemma \ref{lem:triv_actn}]
Given $[f] = [f \colon N \to M_{r, g}] \in \sS^s(M_{r, g})$ and
$\rho \in L^s_{4k+3}(\Z[\Z^r])$ we write $\rho[f \colon N \to M_{r, g}]$
for the action of $\rho$ on $[f]$ by Wall realisation
and $\rho M_{r, g}$ for the source of  $\rho [\id \colon M_{r, g} \to M_{r, g}]$.
It is easy to check that the source of $\rho[f \colon N \to M_{r, g}]$ is $\rho' N$
for some $\rho' \in L_{4k+3}^s(\Z[\Z^r])$.
By Lemma \ref{lem:saut} and the exactness of the surgery exact sequence,
if follows that every structure $[f] \in \sS^s(M_{r, g})$ is equivalent
to $\rho[g \colon M_{r, g} \to M_{r, g}]$ for some $g \in \sAut(M_{r, g})$
and some $\rho \in L_{4k+3}^s(\Z[\Z^r])$.
Hence it suffices to prove that $\rho M_{r, g} \cong M_{r, g}$ for all
$\rho \in L_{4k+3}^s(\Z[\Z^r])$.

Since $M_{r, g}$ is a smooth manifold and Wall realisation is a smooth
operation, we shall work in the smooth category and show that
$\rho M_{r, g}$ is diffeomorphic to $M_{r, g}$.
The definition of Wall realisation ensures the $M_{r, g}$ and $\rho M_{r, g}$
are stably diffeomorphic; i.e.~there is a diffeomorphism
$$ M_{r, g} \sharp_t (S^{2k+1} \times S^{2k+1})
\cong \rho M_{r, g} \sharp_t (S^{2k+1} \times S^{2k+1})$$
for some non-negative integer $t$.
Since $M_{r, g} = M_{r, 0} \sharp_g (S^{2k+1} \times S^{2k+1})$
and $g \geq r{+}3$, $M_{r, g}$ and $\rho M_{r, g}$ are
diffeomorphic by \cite[Theorem 1.1]{Crowley&Sixt(2011)}.
\end{proof}

\begin{rem}
The proof of Lemma \ref{lem:triv_actn} uses the same basic ideas
as the proof of \cite[Theorem E]{Kreck(1999)}.
Besides being in the topological category,
the main difference is that the fundamental groups for Lemma \ref{lem:triv_actn}
are infinite while those covered in \cite[Theorem E]{Kreck(1999)} are finite.
\end{rem}

\begin{proof}[Proof of Theorem~\ref{thmC}]
We assume $r \geq 3$ and so $|\sS(M_{r, g})| = \infty$ by Lemma \ref{lem:imL}.
Lemma \ref{lem:saut} states that the image of $\sS(M)$ in $\sN(M)$
is realised by $\pi_1$-polarised self-equivalences of $M$.  Since
we assume $g \geq r{+}3$,  Lemma \ref{lem:triv_actn} gives
that the action of $L_{4k+3}(\Z[\Z^r])$ on $\sS(M_{r, g})$ does
not change the homeomorphism type of the source of a structure.
Moreover the action of the $L$-group preserves the identification of fundamental
groups in the source and target.
The exactness of the surgery exact sequence for $M_{r, g}$ then ensures
that every structure $[f \colon N \to M_{r, g}]$ has a base-point preserving representative
$g \colon M_{r, g} \to M_{r, g}$, where $g_* = \id \colon \pi_1(M_{r, g}) \to \pi_1(M_{r, g})$;
i.e.~$|\sM_{\pi}(M_{r, g})| = 1$.
\end{proof}

\begin{rem}
For certain finite abelian groups $T$, we believe the arguments above
can be adapted to find manifolds $M$ in all even dimensions $2j \geq 6$ where
$2j \not\equiv 0$~$mod$ $8$, with $\pi_1(M) = \Z \times T$ and
$|\sS^s(M)| = \infty$ but $|\sM^s_\pi(M)| < \infty$.
Whether there are odd-dimensional or $8j$-dimensional manifolds $M$
with this property is an interesting question.
\end{rem}

We conclude with the proof of Theorem \ref{thmE} for which we require
a preparatory lemma. For $k \geq 0$ denote by $\pi_k^S = \colim_{i \to \infty}\pi_{i+k}(S^i)$ the
stable $k$-stem, by $S \colon \pi_{i+k}(S^i) \to \pi_k^S$ the stabilisation map
and for $k \geq 1$ by $\phi \colon \pi_k^S \to \pi_k(\G)$ the canonical isomorphism
as defined in \cite[Corollary 3.8]{Madsen-Milgram(1979)}.

\begin{lem} \label{lem:Adams}
Let $j \colon \G \to \G/\OO$ be the canonical map.
For all $k \geq 1$ the composition
$$ \pi_{4k+2}(S^{2k+1}) \xra{~S~} \pi_{2k+1}^S \xra{~\phi~}
\pi_{2k+1}(\G) \xra{~j_*~}
\pi_{2k+1}(\G/\OO)$$
is onto.
\end{lem}

\begin{proof}
By Bott's computation of $\pi_*(\OO)$ \cite{Bott(1959)} and
a theorem of Adams' \cite[Theorem 1.1]{Adams(1966)},
the canonical map $\pi_{2k}(\OO) \to \pi_{2k}(\G)$ is injective.
From the exact sequence
$$ \pi_{2k+1}(\G) \to \pi_{2k+1}(\G/\OO) \to \pi_{2k}(\OO) \to \pi_{2k}(\G),$$
we see that $\pi_{2k+1}(\G) \to \pi_{2k+1}(\G/\OO)$ is onto. Recall the Hopf-invariant
$H \colon \pi_{4k+1}^S \to \Z/2$.
The EHP sequence \cite[Theorem 2.2, Chapter XII]{Whitehead(1978)}
shows that the suspension map
$S \colon \pi_{4k+2}(S^{2k+1}) \to \pi_{2k+1}^S$
has image all elements with zero Hopf-invariant.
By another theorem of Adams' \cite[Theorem 1.1.1.]{Adams(1960)}
there are only elements with non-zero Hopf-invariant
when $2k{+}1 = 1, 3, 7$,
in which case the image of $\pi_{2k+1}(\OO) \to \pi_{2k+1}(\G)$
contains elements
with non-zero Hopf-invariant.
It follows that every coset of the image of $\pi_{2k+1}(\OO)$ in $\pi_{2k+1}(\G)$
contains an element of trivial Hopf-invariant and so
$j_* \circ \phi \circ S$ is onto.
\end{proof}

\begin{proof}[Proof of Theorem~\ref{thmE}]
The proof of is very similar to the proof of Theorem \ref{thmC}.
We discuss the modifications needed in the smooth case
and specifically the smooth versions of
Lemmas \ref{lem:imL}, \ref{lem:saut} and \ref{lem:triv_actn}.
Consider the following commutative diagram relating the smooth and topological
surgery exact sequences for $M_{r, g, \alpha}$.
\begin{equation} \begin{split}
	\label{eq:ses_Mrg}
	\xymatrix{
	\sN_{\Diff}(M_{r, g, \alpha}) \ar[r]^{\theta_{4k+3}} \ar[d] ^{F_\sN}&
	L^{s}_{4k+3} (\ZZ\pi) \ar[r]^(0.475){\omega} \ar[d]^{=} &
	\sS^{s}_{\Diff} (M_{r, g, \alpha}) \ar[r]^(0.475){\eta} \ar[d]^{F_\sS} &
	\sN_{\Diff} (M_{r, g, \alpha}) \ar[d]^{F_\sN} \ar[r]^(0.5){\theta_{4k+2}} &
	L^{s}_{4k+2} (\ZZ\pi) \ar[d]^{=} \\
	\sN(M_{r, g}) \ar[r]_{\theta_{4k+3}} &
	L^{s}_{4k+3} (\ZZ\pi) \ar[r]_{\omega} &
	\sS^{s}(M_{r, g}) \ar[r]_(0.525){\eta} &
	\sN(M)  \ar[r]_(.45){\theta_{4k+2}} &
	L^{s}_{4k+2} (\ZZ\pi)
	}
\end{split}
\end{equation}
Since the image of $\omega \colon L^s_{4k+3}(\Z[\Z^r])$ is infinite in
$\sS^{s}(M_{r, g})$ it follows immediately that the image of
$L^s_{4k+3}(\Z[\Z^r])$ is infinite in $\sS^s_{\Diff}(M_{r, g, \alpha})$.
We compute the normal invariants
using the fact from Lemma \ref{lem:M_{r,g}} that $M_{r, g}$
has the stable homotopy type of a wedge of spheres
and \eqref{eq:ses_normal_square}.
We obtain a commutative diagram,
\begin{equation}
\begin{array}{ccc}
\sN_{\Diff}(M_{r, g, \alpha}) \!\!\!\!\!
& \xra{\equiv} &
\!\!\!\!\bigoplus_s \! \pi_2(\G/\OO)
\bigoplus_{2g} \! \pi_{2k+1}(\G/\OO)
\bigoplus_s \!\pi_{4k}(\G/\OO)
\bigoplus_r \! \pi_{4k+1}(\G/\OO)
\bigoplus \! \pi_{4k+2}(\G/\OO)
\\
\downarrow_{F_\sN} & &  \downarrow_{\,i_*} \\
\sN(M_{r, g}) & \xra{\equiv} &
\bigoplus_s \! \pi_2(\G/\TOP)
\bigoplus_s \! \pi_{4k}(\G/\TOP)
\bigoplus \! \pi_{4k+2}(\G/\TOP)
\end{array}
\end{equation}
where $i_*$ is induced by the canonical map $\G/\OO \to \G/\TOP$.
In the proof of Theorem \ref{thmC} we showed
for the topological normal invariants
that the summands
$\bigoplus_s\pi_{4k}(\G/\TOP)$ and $\pi_{4k+2}(\G/\TOP)$
map injectively to $L^s_{4k+2}(\Z[\Z^r])$.
The proof of Theorem \ref{thmC} also showed
that all the elements of the summand $\bigoplus_s\pi_2(\G/\TOP)$
are realised as the normal invariants of pinch maps
on $M_{r, g}$.
Using Lemma \ref{lem:Adams}, similar arguments with pinch maps along the inclusions
$$S^{2k+1} \times \{\ast\} \subset ((S^{2k+1} \times S^{2k+1}) - \int(D^{4k+2}))
\subset M_{r, g, \alpha}$$
show that all the elements of the summand $\bigoplus_{2g}\pi_{2k+1}(\G/\OO)$
are realised as the normal invariants of pinch maps
on $M_{r, g}$.
From the smooth surgery exact sequence of the sphere and
the work Kervaire and Milnor \cite{Kervaire-Milnor(1963)},
for $6 \leq m = 4k, 4k+2$ there are isomorphisms
$$ \Theta_{m} \cong \ker \bigl( \pi_{m}(\G/\OO) \to \pi_{m}(\G/\TOP) \bigr),$$
and also isomorphisms
$$ \Theta_{4k+1}/bP_{4k+2} \cong \pi_{4k+1}(\G/\OO).$$
Since the space $\TOP/\OO$ is $6$-connected~\cite[Essay V.5]{Kirby-Siebenmann(1977)}, we also have that the map $\pi_4(\G/\OO) \to \pi_4(\G/\TOP)$ is injective. The points above combine to complete the proof.
\end{proof}




%

\small
\bibliography{CM}

\def\cprime{$'$}
\begin{thebibliography}{{Kha}17}

\bibitem[Ada60]{Adams(1960)}
J.~F. Adams.
\newblock On the non-existence of elements of {H}opf invariant one.
\newblock {\em Ann. of Math. (2)}, 72:20--104, 1960.

\bibitem[Ada66]{Adams(1966)}
J.~F. Adams.
\newblock On the groups {$J(X)$}. {IV}.
\newblock {\em Topology}, 5:21--71, 1966.

\bibitem[BDK07]{Brookman-Davis-Khan(2007)}
Jeremy {Brookman}, James~F. {Davis}, and Qayum {Khan}.
\newblock {Manifolds homotopy equivalent to $P^{n} \# P^{n}$.}
\newblock {\em {Math. Ann.}}, 338(4):947--962, 2007.

\bibitem[BHS64]{Bass-Heller-Swan(1964)}
Hyman Bass, Alex Heller, and Richard~G. Swan.
\newblock The {W}hitehead group of a polynomial extension.
\newblock {\em Inst. Hautes \'Etudes Sci. Publ. Math.}, 22:61--79, 1964.

\bibitem[Bot57]{Bott(1959)}
Raoul Bott.
\newblock The stable homotopy of the classical groups.
\newblock {\em Proc. Nat. Acad. Sci. U.S.A.}, 43:933--935, 1957.

\bibitem[CH15]{Crowley&Hambleton(2015)}
Diarmuid Crowley and Ian Hambleton.
\newblock Finite group actions on {K}ervaire manifolds.
\newblock {\em Adv. Math.}, 283:88--129, 2015.

\bibitem[CS11]{Crowley&Sixt(2011)}
Diarmuid Crowley and J\"org Sixt.
\newblock Stably diffeomorphic manifolds and {$l_{2q+1}(\Bbb Z[\pi])$}.
\newblock {\em Forum Math.}, 23(3):483--538, 2011.

\bibitem[CW03]{Chang&Weinberger(2003)}
Stanley Chang and Shmuel Weinberger.
\newblock On invariants of {H}irzebruch and {C}heeger-{G}romov.
\newblock {\em Geom. Topol.}, 7:311--319, 2003.

\bibitem[JK11]{Jahren-Kwasik(2008)}
Bj{\o}rn {Jahren} and S{\l}awomir {Kwasik}.
\newblock {Free involutions on $S ^{1} \times S ^{n}$.}
\newblock {\em {Math. Ann.}}, 351(2):281--303, 2011.

\bibitem[{Kha}17]{Khan(2017)}
Qayum {Khan}.
\newblock {Free transformations of $S ^{1} \times S ^{n }$ of square-free odd
  period}.
\newblock {\em {Indiana Univ. Math J.}}, 66(5):1453--1482, 2017.

\bibitem[KL09]{Kreck&Lueck(2009)}
M.~Kreck and W.~L{\"u}ck.
\newblock Topological rigidity for non-aspherical manifolds.
\newblock {\em Pure Appl. Math. Q.}, 5(3, Special Issue: In honor of Friedrich
  Hirzebruch. Part 2):873--914, 2009.

\bibitem[KM63]{Kervaire-Milnor(1963)}
Michel~A. Kervaire and John~W. Milnor.
\newblock Groups of homotopy spheres. {I}.
\newblock {\em Ann. of Math. (2)}, 77:504--537, 1963.

\bibitem[Kre99]{Kreck(1999)}
Matthias Kreck.
\newblock Surgery and duality.
\newblock {\em Ann. of Math. (2)}, 149(3):707--754, 1999.

\bibitem[KS77]{Kirby-Siebenmann(1977)}
Robion~C. Kirby and Laurence~C. Siebenmann.
\newblock {\em Foundational essays on topological manifolds, smoothings, and
  triangulations}.
\newblock Princeton University Press, Princeton, N.J., 1977.
\newblock With notes by John Milnor and Michael Atiyah, Annals of Mathematics
  Studies, No. 88.

\bibitem[LR05]{Lueck-Reich(2005)}
Wolfgang L{\"u}ck and Holger Reich.
\newblock The {B}aum-{C}onnes and the {F}arrell-{J}ones conjectures in {$K$}-
  and {$L$}-theory.
\newblock In {\em Handbook of $K$-theory. Vol. 1, 2}, pages 703--842. Springer,
  Berlin, 2005.

\bibitem[L{\"u}c02]{Lueck(2001)}
Wolfgang L{\"u}ck.
\newblock A basic introduction to surgery theory.
\newblock In {\em Topology of high-dimensional manifolds, No. 1, 2 (Trieste,
  2001)}, volume~9 of {\em ICTP Lect. Notes}, pages 1--224. Abdus Salam Int.
  Cent. Theoret. Phys., Trieste, 2002.

\bibitem[MM79]{Madsen-Milgram(1979)}
Ib~Madsen and R.~James Milgram.
\newblock {\em The classifying spaces for surgery and cobordism of manifolds},
  volume~92 of {\em Annals of Mathematics Studies}.
\newblock Princeton University Press, Princeton, N.J., 1979.

\bibitem[MS60]{Milnor-Spanier(1960)}
John Milnor and Edwin Spanier.
\newblock Two remarks on fiber homotopy type.
\newblock {\em Pacific J. Math.}, 10:585--590, 1960.

\bibitem[MTW80]{Madsen-Taylor-Williams(1980)}
Ib~Madsen, Laurence~R. Taylor, and Bruce Williams.
\newblock Tangential homotopy equivalences.
\newblock {\em Comment. Math. Helv.}, 55(3):445--484, 1980.

\bibitem[Nov66]{Novikov(1966)}
S.~P. Novikov.
\newblock On manifolds with free abelian fundamental group and their
  application.
\newblock {\em Izv. Akad. Nauk SSSR Ser. Mat.}, 30:207--246, 1966.

\bibitem[Ran92]{Ranicki(1992)}
A.~A. Ranicki.
\newblock {\em Algebraic {$L$}-theory and topological manifolds}, volume 102 of
  {\em Cambridge Tracts in Mathematics}.
\newblock Cambridge University Press, Cambridge, 1992.

\bibitem[Ran09]{Ranicki(2009)}
Andrew Ranicki.
\newblock A composition formula for manifold structures.
\newblock {\em Pure Appl. Math. Q.}, 5(2, part 1):701--727, 2009.

\bibitem[Sha69]{Shaneson(1969)}
Julius~L. Shaneson.
\newblock Wall's surgery obstruction groups for {$G\times Z$}.
\newblock {\em Ann. of Math. (2)}, 90:296--334, 1969.

\bibitem[Wal66]{Wall(1966)}
C.~T.~C. Wall.
\newblock Classification problems in differential topology. {IV}.
  {T}hickenings.
\newblock {\em Topology}, 5:73--94, 1966.

\bibitem[Wal99]{Wall(1999)}
C.~T.~C. Wall.
\newblock {\em Surgery on compact manifolds}, volume~69 of {\em Mathematical
  Surveys and Monographs}.
\newblock American Mathematical Society, Providence, RI, second edition, 1999.
\newblock Edited and with a foreword by A. A. Ranicki.

\bibitem[Wei90]{Weinberger(1990)}
Shmuel Weinberger.
\newblock On smooth surgery.
\newblock {\em Comm. Pure Appl. Math.}, 43(5):695--696, 1990.

\bibitem[Whi78]{Whitehead(1978)}
George~W. Whitehead.
\newblock {\em Elements of homotopy theory}, volume~61 of {\em Graduate Texts
  in Mathematics}.
\newblock Springer-Verlag, New York-Berlin, 1978.

\bibitem[WXY17]{Weinberger-Xie-Yu(2017)}
Shmuel Weinberger, Zhizhang Xie, and Guoliang Yu.
\newblock Additivity of higher rho invariants and nonrigidity of topological
  manifolds.
\newblock {\em arXiv:1608.03661}, 2017.

\bibitem[WY15]{Weinberger-Yu(2015)}
Shmuel {Weinberger} and Guoliang {Yu}.
\newblock {Finite part of operator $K$-theory for groups finitely embeddable
  into Hilbert space and the degree of nonrigidity of manifolds.}
\newblock {\em {Geom. Topol.}}, 19(5):2767--2799, 2015.

\end{thebibliography}
\bibliographystyle{alpha}

\end{document}